\documentclass[10pt,a4paper]{amsart}
\usepackage{amssymb,amsmath,amsfonts}

\headsep=1.2500cm \numberwithin{equation}{section}
\newcommand{\set}[1]{\left\{#1\right\}}
\newcommand{\A}{UP(I,\mathcal{A})}
\newtheorem{Theorem}{Theorem}[section]
\newtheorem{Proposition}[Theorem]{Proposition}
\newtheorem{cor}[Theorem]{Corollary}
\newtheorem{lemma}[Theorem]{Lemma}
\theoremstyle{remark}
\newtheorem{Definition}[Theorem]{Definition}
\newtheorem{Example}[Theorem]{Example}

\newtheorem{Remark}[Theorem]{Remark}

\begin{document}
\title{Johnson pseudo-Connes amenability of dual Banach algebras}

\author[S. F. Shariati]{S. F. Shariati}

\address{Faculty of Mathematics and Computer Science,
	Amirkabir University of Technology, 424 Hafez Avenue, 15914
	Tehran, Iran.}

\email{f.shariati@aut.ac.ir}

\author[A. Pourabbas]{A. Pourabbas}
\email{arpabbas@aut.ac.ir}

\author[A. Sahami]{A. Sahami}
\address{Department of Mathematics Faculty of Basic Science, Ilam University, P.O. Box 69315-516 Ilam, Iran.}
\email{amir.sahami@aut.ac.ir}



%\maketitle
%-----------------------------------------------------------------------------------------
%%%%%%%%%%%%%%%%%%%%%%%%%%%%%%%%%%%%%%%%%%%%%
\begin{abstract}
We introduce the notion of Johnson pseudo-Connes amenability for dual Banach algebras. We study the relation between this new notion with the various notions of Connes amenability like Connes amenability, approximate Connes amenability and pseudo Connes amenability. We also investigate some hereditary properties of this new notion.   We prove that for a locally compact group $G$, $M(G)$ is Johnson pseudo-Connes amenable if and only if $G$ is amenable. Also we show that for every non-empty set $I$, $\mathbb{M}_I(\mathbb{C})$ under this new notion is forced to have a finite index. Finally, we provide some examples of certain dual Banach algebras and we study their Johnson pseudo-Connes amenability. 
\end{abstract}
\maketitle
%-----------------------------------------------------------------------------------------
\section{Introduction and Preliminaries}
The concept of amenability for a Banach algebra was first introduced by B. E. Johnson \cite{Runde:2002}. A Banach algebra $\mathcal{A}$ is amenable if and only if $\mathcal{A}$ has a virtual diagonal, that is, there exists an element $M$ in $(\mathcal{A}\hat{\otimes}\mathcal{A})^{**}$ such that $a\cdot M=M\cdot a$ and $\pi_{\mathcal{A}}^{**}(M)a=a$, where $\pi_{\mathcal{A}}:\mathcal{A}\hat{\otimes}\mathcal{A}\rightarrow\mathcal{A}$ is a bounded $\mathcal{A}$-bimodule morphism defined by $\pi_{\mathcal{A}}(a\otimes{b})=ab$ for every $a,b\in{\mathcal{A}}$. 

The class of dual Banach algebras were introduced by Runde \cite{Runde:2001}. Let $\mathcal{A}$ be a Banach algebra and let $E$ be a Banach $\mathcal{A}$-bimodule. An $\mathcal{A}$-bimodule $E$ is called dual if there exists  a closed submodule ${E}_{\ast}$ of ${E}^{\ast}$ such that $E=(E_{\ast})^{\ast}$. The Banach algebra $\mathcal{A}$ is called dual if it is dual as a Banach $\mathcal{A}$-bimodule. The measure algebra $M(G)$ of a locally compact group $G$, the algebra of bounded operators $\mathcal{B}(E)$, for a reflexive Banach space $E$ and the second dual $\mathcal{A}^{\ast\ast}$ of an Arens regular Banach algebra $\mathcal{A}$ are examples of dual Banach algebras.

For a given dual Banach algebra $\mathcal{A}$ and a Banach $\mathcal{A}$-bimodule $E$, we denote by $\sigma{wc}(E)$, the set of all elements $x\in{E}$ such that the module maps $\mathcal{A}\rightarrow{E}$; ${a}\mapsto{a}\cdot{x}$ and ${a}\mapsto{x}\cdot{a}$
are $wk^\ast$-$wk$-continuous, which is a closed submodule of $E$. Since $\sigma{wc}(\mathcal{A}_{\ast})=\mathcal{A}_{\ast}$, the adjoint of $\pi_\mathcal{A}$ maps $\mathcal{A}_{\ast}$ into $\sigma{wc}(\mathcal{A}\hat{\otimes}\mathcal{A})^{\ast}$. Therefore, $\pi_\mathcal{A}^{\ast\ast}$ drops to an $\mathcal{A}$-bimodule morphism $\pi_{\sigma{wc}}:(\sigma{wc}(\mathcal{A}\hat{\otimes}\mathcal{A})^{\ast})^{\ast}\rightarrow\mathcal{A}$. A dual Banach algebra $\mathcal{A}$ is called Connes amenable if and only if $\mathcal{A}$ has an $\sigma{wc}$-virtual diagonal, that is, there exists an element $M\in{(\sigma{wc}(\mathcal{A}\hat{\otimes}\mathcal{A})^{*})^*}$ such that	$a\cdot{M}=M\cdot{a}$ and ${a}\pi_{\sigma{wc}}(M)=a$ for every $a\in{\mathcal{A}}$ \cite{Runde:2004}. Some new generalizations of Connes amenability like approximate Connes amenability and pseudo-Connes amenability have been introduced  by Esslamzadeh \cite{Eslam:2012} and  Mahmoodi 
 \cite{Mahmoodi:14}. A unital dual Banach algebra $\mathcal{A}$ is approximate Connes amenable if and only if there exists a net $(M_{\alpha})$ in ${(\sigma{wc}(\mathcal{A}\hat{\otimes}\mathcal{A})^{*})^*}$ such that 
$a\cdot M_{\alpha}-M_{\alpha}\cdot a\rightarrow0$ and $\pi_{\sigma{wc}}(M_{\alpha})a\rightarrow a$ for every $a\in{\mathcal{A}}$ \cite[Theorem 3.3]{Eslam:2012}. Also the dual Banach algebra $\mathcal{A}$ is called pseudo-Connes amenable if there exists a net $(M_{\alpha})$ in $\mathcal{A}\hat{\otimes}\mathcal{A}$ such that $a\cdot M_{\alpha}-M_{\alpha}\cdot a\overset{wk^*}{\rightarrow}0$ in ${(\sigma{wc}(\mathcal{A}\hat{\otimes}\mathcal{A})^{*})^*}$ and $\pi_{\sigma{wc}}(M_{\alpha})a\overset{wk^*}{\rightarrow}a$ in $\mathcal{A}$ \cite[Definition 4.3]{Mahmoodi:14}.

The notion of Johnson pseudo-contractibility for a Banach algebra was introduced by second and third authors, which is a weaker notion than amenability and pseudo-contractibility but it is stronger than pseudo-amenability \cite{Sahami:2017}. A Banach algebra $\mathcal{A}$ is called Johnson pseudo-contractible, if there exists a not necessarily bounded net $(m_{\alpha})$ in $(\mathcal{A}\hat{\otimes}\mathcal{A})^{**}$ such that $a\cdot{m_{\alpha}}={m_{\alpha}}\cdot{a}$ and $\pi_{\mathcal{A}}^{**}(m_{\alpha}){a}\rightarrow{a}$ for every $a\in{\mathcal{A}}$. They also  showed that for a 
locally compact group $G$, $M(G)$ is Johnson pseudo-contractible if and only if $G$ is discrete and amenable \cite[Proposition 3.3]{Sahami:2017}. They characterized the Johnson pseudo-contractibility of $\ell^1(S)$, where $S$ is a uniformly locally finite inverse semigroup \cite[Theorem 2.3]{saham:17}. They showed that for a Brandt semigroup $S=M^0(G,I)$ over a non-empty set $I$, $\ell^1(S)$ is Johnson pseudo-contractible if and only if $G$ is amenable and $I$ is finite \cite[Theorem 2.4]{saham:17}.

Motivated by these results, we introduce the new notion for dual Banach algebras, which is a weaker notion than Connes amenability and it is a stronger notion  than approximate Connes amenability and pseudo-Connes amenability.
\begin{Definition}	
	A dual Banach algebra $\mathcal{A}$ is called Johnson pseudo-Connes amenable, if there exists a not necessarily bounded net $(m_{\alpha})$ in $(\mathcal{A}\hat{\otimes}\mathcal{A})^{**}$ such that $\langle{T},a\cdot{m_{\alpha}}\rangle=\langle{T},{m_{\alpha}}\cdot{a}\rangle$ and $i^{\ast}_{\mathcal{A}_{\ast}}\pi_{\mathcal{A}}^{**}(m_{\alpha}){a}\rightarrow{a}$ for every $a\in{\mathcal{A}}$ and $T\in\sigma{wc}(\mathcal{A}\hat{\otimes}\mathcal{A})^{\ast}$, where $i_{\mathcal{A}_{\ast}}:\mathcal{A}_{\ast}\hookrightarrow\mathcal{A}^{\ast}$ is the canonical embedding.
\end{Definition}
It is clear that every Johnson pseudo-contractible dual Banach algebra is Johnson pseudo-Connes amenable.

 In this paper we investigate the relation between a new notion of Johnson pseudo-Connes amenability to various notions of Connes amenability.  Indeed we show that for a dual Banach algebra $\mathcal{A}$, Johnson pseudo-Connes amenability of $\mathcal{A}$ implies $\varphi$-Connes amenability, where $\varphi$ is a ${wk}^{\ast}$-continuous character on $\mathcal{A}$. Using this tool we show that for a finite set $I$ and a dual Banach algebra $\mathcal{A}$ with the non-empty $wk^*$-continuous character space, a class of $I\times{I}$-upper triangular matrix $UP(I,\mathcal{A})$ is Johnson pseudo-Connes amenable if and only if $\mathcal{A}$ is Johnson pseudo-Connes amenable and $I$ is singleton.  
 
 As an application we prove that for the arbitrary set $I$, the Banach algebra of $I\times I$-matrices over $\mathbb{C}$, $M_{I}(\mathbb{C})$ is Johnson pseudo-Connes amenable if and only if $I$ is finite. Also we show that for a locally compact group $G$, $M(G)$ is Johnson pseudo-Connes amenable if and only if $G$ is amenable. This result distinguishes our new notion with Johnson pseudo-contractibility. Finally, we provide some examples of certain dual Banach algebras and we study their Johnson pseudo-Connes amenability.

%-----------------------------------------------------------------------------------------
%%%%%%%%%%%%%%%%%%%%%%%%%%%%%%%%%%%%%%%%%%%%%
\section{Johnson pseudo-Connes amenability} 

 For a Banach algebra  $\mathcal{A}$, the projective tensor product $\mathcal{A}\hat{\otimes}\mathcal{A}$ is a Banach $\mathcal{A}$-bimodule via the following actions
\begin{equation*}
a\cdot(b\otimes c)=ab\otimes c,\qquad(a\otimes b)\cdot c=a\otimes bc\quad(a,b,c\in{\mathcal{A}}).
\end{equation*}
If $E$ is a Banach $\mathcal{A}$-bimodule, then $E^*$ is also a Banach $\mathcal{A}$-bimodule via the following actions
\begin{equation*}
(a\cdot f)(x)=f(x\cdot a),\qquad(f\cdot a)(x)=f(a\cdot x)\quad(a\in{\mathcal{A}},x\in{E},f\in{E^*}).
\end{equation*}
\begin{Remark}\label{qou}
Let $\mathcal{A}$ be a dual Banach algebra and let $E$ be a Banach $\mathcal{A}$-bimodule. Since $\sigma wc(E^*)$ is a closed $\mathcal{A}$-submodule of $E^{*}$, we have a quotient map $q:E^{\ast\ast}\longrightarrow\sigma{wc}(E^{\ast})^{\ast}$ is defined by $q(M)=M\vert_{\sigma wc(E^*)}$ for every $M\in{E^{**}}$.
\end{Remark}
\begin{Remark}\label{R1}
	Let $\mathcal{A}$ be a dual Banach algebra. Then for every $M\in{(\mathcal{A}\hat{\otimes}\mathcal{A})^{**}}$ and $f\in{\mathcal{A}_{\ast}}$ we have
	\begin{equation*}
	\begin{split}
	\langle f,\pi_{\sigma{wc}}q(M)\rangle&=\langle \pi^*\vert_{\mathcal{A}_{\ast}}(f),q(M)\rangle=\langle \pi^*\vert_{\mathcal{A}_{\ast}}(f),M\rangle=\langle\pi^*(f),M\rangle\\&=\langle f,\pi^{**}(M)\rangle=\langle i_{\mathcal{A}_{\ast}}(f),\pi^{**}(M)\rangle=\langle f,i^{\ast}_{\mathcal{A}_{\ast}}\pi^{**}(M)\rangle,
	\end{split}
	\end{equation*}
where $q:(\mathcal{A}\hat{\otimes}\mathcal{A})^{\ast\ast}\longrightarrow(\sigma{wc}({\mathcal{A}}\hat{\otimes}{\mathcal{A}})^{\ast})^{\ast}$ is as above. So $i^{\ast}_{\mathcal{A}_{\ast}}\pi^{**}_{\mathcal{A}}=\pi_{\sigma{wc}}q$.
\end{Remark}
\begin{lemma}\label{l3}
	Let $\mathcal{A}$ be a dual Banach algebra. If $\mathcal{A}$ is Connes amenable, then $\mathcal{A}$ is Johnson pseudo-Connes amenable.
\end{lemma}	
\begin{proof}
Let $\mathcal{A}$ be a Connes amenable Banach algebra. Then by \cite[Theorem 4.8]{Runde:2004}, there is an element $\tilde{M}\in{(\sigma{wc}(\mathcal{A}\hat{\otimes}\mathcal{A})^{\ast})^*}$ such that
\begin{center}
	$a\cdot{\tilde{M}}=\tilde{M}\cdot{a}\quad$ and $\quad\pi_{\sigma{wc}}(\tilde{M})a=a\quad(a\in{\mathcal{A}})$. 
\end{center}
Consider $M\in{(\mathcal{A}\hat{\otimes}\mathcal{A})^{**}}$ such that $q(M)=\tilde{M}$, where $q:(\mathcal{A}\hat{\otimes}\mathcal{A})^{\ast\ast}\longrightarrow(\sigma{wc}({\mathcal{A}}\hat{\otimes}{\mathcal{A}})^{\ast})^{\ast}$ is the quotient map as in the Remark \ref{qou}.
By Remark \ref{R1}, for every $a\in{\mathcal{A}}$ and $T\in\sigma{wc}(\mathcal{A}\hat{\otimes}\mathcal{A})^{\ast}$ we have
\begin{center}
	$\langle{T},a\cdot{M}\rangle=\langle{T},{M}\cdot{a}\rangle\quad$ and $\quad i^{\ast}_{\mathcal{A}_{\ast}}\pi_{\mathcal{A}}^{**}(M){a}={a}$. 
\end{center}
\end{proof}
\begin{lemma}
	Let $\mathcal{A}$ be a dual Banach algebra. If $\mathcal{A}$ is Johnson pseudo-Connes amenable, then $\mathcal{A}$ is pseudo-Connes amenable.
\end{lemma}
\begin{proof}
Since $\mathcal{A}$ is Johnson pseudo-Connes amenable, there exists a net $(m_{\alpha})$ in $(\mathcal{A}\hat{\otimes}\mathcal{A})^{**}$ such that $\langle{T},a\cdot{m_{\alpha}}\rangle=\langle{T},{m_{\alpha}}\cdot{a}\rangle$ and $i^{\ast}_{\mathcal{A}_{\ast}}\pi_{\mathcal{A}}^{**}(m_{\alpha}){a}\rightarrow{a}$ for every $a\in{\mathcal{A}}$ and $T\in\sigma{wc}(\mathcal{A}\hat{\otimes}\mathcal{A})^{\ast}$. By Goldstein's theorem, there is a net $(u_\beta^\alpha)_{\beta\in{\Theta}}$ in $\mathcal{A}\hat{\otimes}\mathcal{A}$ such that $wk^*\hbox{-}\lim\limits_{\beta}{u}_\beta^\alpha=m_{\alpha}$. Thus for every $a\in{\mathcal{A}}$
\begin{equation*}
wk^*\hbox{-}\lim\limits_{\beta}a\cdot u_\beta^\alpha-u_\beta^\alpha\cdot a=a\cdot m_{\alpha}-m_{\alpha}\cdot a.
\end{equation*} 
Since $i^{\ast}_{\mathcal{A}_{\ast}}\pi_{\mathcal{A}}^{**}$ is $wk^*$-continuous and the multiplication in $\mathcal{A}$ is separately $wk^*$-continuous \cite[Exercise 4.4.1]{Runde:2002}, for every $a\in{\mathcal{A}}$ we have
\begin{equation*}
wk^*\hbox{-}\lim\limits_{\alpha}wk^*\hbox{-}\lim\limits_{\beta}i^{\ast}_{\mathcal{A}_{\ast}}\pi_{\mathcal{A}}^{**}(u_\beta^\alpha)a=wk^*\hbox{-}\lim\limits_{\alpha}i^{\ast}_{\mathcal{A}_{\ast}}\pi_{\mathcal{A}}^{**}(m_{\alpha})a=a.
\end{equation*}
Let $E=I\times\Theta^I$ be a directed set with product ordering defined by
\begin{equation*}
(\alpha,\beta)\leq_E(\alpha^\prime,\beta^\prime)\Leftrightarrow\alpha\leq_I\alpha^\prime\,\, {\hbox{and}}\,\,\,\beta\leq_{\Theta^I}\beta^\prime\qquad(\alpha,\alpha^\prime\in{I},\quad \beta,\beta^\prime\in{\Theta^I}),
\end{equation*}
where $\Theta^I$ is the set of all functions from $I$ into $\Theta$ and $\beta\leq_{\Theta^I}\beta^\prime$ means that $\beta(d)\leq_\Theta\beta^\prime(d)$ for every $d\in{I}$. Suppose that $\gamma=(\alpha,\beta_\alpha)$ and $n_\gamma=u_\beta^\alpha$. By iterated limit theorem \cite[Page 69]{kelley:75}, one can see that $wk^*\hbox{-}\lim\limits_{\gamma}a\cdot n_\gamma-n_\gamma\cdot a=0$ in $(\sigma{wc}(\mathcal{A}\hat{\otimes}\mathcal{A})^{\ast})^*$ and by Remark \ref{R1}, $wk^*\hbox{-}\lim\limits_{\gamma}\pi_{\sigma{wc}}q(n_\gamma)a=a$ in $\mathcal{A}$, where $q:(\mathcal{A}\hat{\otimes}\mathcal{A})^{\ast\ast}\longrightarrow(\sigma{wc}({\mathcal{A}}\hat{\otimes}{\mathcal{A}})^{\ast})^{\ast}$ is the quotient map as in the Remark \ref{qou}. So $\mathcal{A}$ is pseudo-Connes amenable \cite[Definition 4.3]{Mahmoodi:14}.
\end{proof}
\begin{lemma}\label{l5}
	Let $\mathcal{A}$ be a unital dual Banach algebra.  If $\mathcal{A}$ is Johnson pseudo-Connes amenable, then $\mathcal{A}$ is approximately Connes amenable.
\end{lemma}
\begin{proof}
Since $\mathcal{A}$ is Johnson pseudo-Connes amenable, there exists a net $(m_{\alpha})$ in $(\mathcal{A}\hat{\otimes}\mathcal{A})^{**}$ such that $\langle{T},a\cdot{m_{\alpha}}\rangle=\langle{T},{m_{\alpha}}\cdot{a}\rangle$ and $i^{\ast}_{\mathcal{A}_{\ast}}\pi_{\mathcal{A}}^{**}(m_{\alpha}){a}\rightarrow{a}$ for every $a\in{\mathcal{A}}$ and $T\in\sigma{wc}(\mathcal{A}\hat{\otimes}\mathcal{A})^{\ast}$. Let $\tilde{m}_\alpha=q(m_{\alpha})$, where $q:(\mathcal{A}\hat{\otimes}\mathcal{A})^{\ast\ast}\rightarrow(\sigma{wc}({\mathcal{A}}\hat{\otimes}{\mathcal{A}})^{\ast})^{\ast}$ is the quotient map as in the Remark \ref{qou}. So $a\cdot\tilde{m}_\alpha=\tilde{m}_\alpha\cdot a$ for every $a\in{\mathcal{A}}$ and by Remark \ref{R1} we have $\pi_{\sigma{wc}}(\tilde{m}_\alpha)a\rightarrow a$. So $\mathcal{A}$ is approximately Connes amenable \cite[Theorem 3.3]{Eslam:2012}.
	\end{proof}
 A dual Banach algebra $\mathcal{A}$ is called Connes biprojective if there exists a bounded $\mathcal{A}$-bimodule morphism $\rho:\mathcal{A}\longrightarrow(\sigma{wc}(\mathcal{A}\hat{\otimes}\mathcal{A})^{\ast})^{\ast}$ such that $\pi_{\sigma{wc}}\circ\rho=id_{\mathcal{A}}$. Shirinkalam and second author \cite{Shi:2016} showed that a dual Banach algebra $\mathcal{A}$ is Connes amenable if and only if $\mathcal{A}$ is Connes biprojective and $\mathcal{A}$  has an identity.
\begin{Proposition}
	Let $\mathcal{A}$ be a dual Banach algebra with a central approximate identity. If $\mathcal{A}$ is Connes biprojective, then $\mathcal{A}$ is Johnson pseudo-Connes amenable.
\end{Proposition}
\begin{proof}
Let $(e_\alpha)$ be a central approximate identity for $\mathcal{A}$ and let $\rho:\mathcal{A}\longrightarrow(\sigma{wc}({\mathcal{A}}\hat{\otimes}{\mathcal{A}})^{\ast})^{\ast}$ be a bounded $\mathcal{A}$-bimodule morphism such that $\pi_{\sigma{wc}}\circ\rho=id_{\mathcal{A}}$. Consider the net $\rho(e_\alpha)$ in $(\sigma{wc}({\mathcal{A}}\hat{\otimes}{\mathcal{A}})^{\ast})^{\ast}$. So there is a net $(m_\alpha)$ in $({\mathcal{A}}\hat{\otimes}{\mathcal{A}})^{**}$ such that $q(m_\alpha)=\rho(e_\alpha)$, where $q:(\mathcal{A}\hat{\otimes}\mathcal{A})^{\ast\ast}\rightarrow(\sigma{wc}({\mathcal{A}}\hat{\otimes}{\mathcal{A}})^{\ast})^{\ast}$ is the quotient map as in the Remark \ref{qou}. For every $a\in{\mathcal{A}}$ and $T\in\sigma{wc}(\mathcal{A}\hat{\otimes}\mathcal{A})^{\ast}$ we have
\begin{equation*}
\begin{split}
\langle{T},a\cdot{m_{\alpha}}\rangle&=\langle T\cdot a,{m_{\alpha}}\rangle=\langle T\cdot a,{m_{\alpha}}\vert_{\sigma{wc}({\mathcal{A}}\hat{\otimes}{\mathcal{A}})^{\ast}}\rangle=\langle T\cdot a,q(m_{\alpha}) \rangle=\langle T\cdot a,\rho(e_\alpha)\rangle=\langle T,a\cdot\rho(e_\alpha)\rangle\\&=\langle T,\rho(ae_\alpha)\rangle=\langle T,\rho(e_\alpha a)\rangle=\langle T,\rho(e_\alpha)\cdot a\rangle=\langle a\cdot T,\rho(e_\alpha)\rangle=\langle a\cdot T,q(m_{\alpha}) \rangle=\langle T,q(m_{\alpha})\cdot a\rangle\\&=\langle T,q(m_{\alpha}\cdot a)\rangle=\langle T,(m_{\alpha}\cdot a)\vert_{\sigma{wc}({\mathcal{A}}\hat{\otimes}{\mathcal{A}})^{\ast}}\rangle=\langle{T},{m_{\alpha}}\cdot{a}\rangle,
\end{split}
\end{equation*}
and by Remark \ref{R1}
\begin{equation*}
i^{\ast}_{\mathcal{A}_{\ast}}\pi_{\mathcal{A}}^{**}(m_{\alpha}){a}=\pi_{\sigma{wc}} q(m_{\alpha}){a}=\pi_{\sigma{wc}}\circ\rho(e_\alpha)a=e_\alpha a\rightarrow a.
\end{equation*}
So $\mathcal{A}$ is Johnson pseudo-Connes amenable.
	\end{proof}

The notion of $\varphi$-Connes amenability for a dual Banach algebra $\mathcal{A}$ introduced by Mahmoodi  and some characterizations were given \cite{Mahmoodi:2014} and \cite{ram:18}, where $\varphi$ is a ${wk}^{\ast}$-continuous character on $\mathcal{A}$. We say that $\mathcal{A}$ is $\varphi$-Connes amenable if there exists a bounded linear functional $m$ on  $\sigma{wc}({\mathcal{A}}^{\ast})$ satisfying $m(\varphi)=1$ and $m(f\cdot{a})=\varphi(a)m(f)$ for every $a\in{\mathcal{A}}$ and $f\in{\sigma{wc}({\mathcal{A}}^{\ast})}$. The set of all ${wk}^{\ast}$-continuous characters on $\mathcal{A}$ is denoted by $\Delta_{wk^*}(\mathcal{A})$.
\begin{Proposition}\label{p6}
	Let $\mathcal{A}$ be a dual Banach algebra and $\varphi\in{\Delta_{wk^*}(\mathcal{A})}$. If $\mathcal{A}$ is Johnson pseudo-Connes amenable, then $\mathcal{A}$ is $\varphi$-Connes amenable.
\end{Proposition}
\begin{proof}
Since $\mathcal{A}$ is Johnson pseudo-Connes amenable, there exists a net $(m_{\alpha})$ in $(\mathcal{A}\hat{\otimes}\mathcal{A})^{**}$ such that $\langle{T},a\cdot{m_{\alpha}}\rangle=\langle{T},{m_{\alpha}}\cdot{a}\rangle$ and $i^{\ast}_{\mathcal{A}_{\ast}}\pi_{\mathcal{A}}^{**}(m_{\alpha}){a}\rightarrow{a}$ for every $a\in{\mathcal{A}}$ and $T\in\sigma{wc}(\mathcal{A}\hat{\otimes}\mathcal{A})^{\ast}$.
Define $\theta:\mathcal{A}\hat{\otimes}\mathcal{A}\rightarrow\mathcal{A}$ by $\theta(a\otimes b)=\varphi(b)a$ for every $a,b\in{\mathcal{A}}$. So for each $a,b,c\in{\mathcal{A}}$,
\begin{equation*}
\theta(a\cdot (b\otimes c))=\theta(ab\otimes c)=\varphi(c)ab=a\theta(b\otimes c),
\end{equation*}
and 
\begin{equation*}
\theta( (b\otimes c)\cdot a)=\theta(b\otimes ca)=\varphi(ca)b=\varphi(c)\varphi(a)b=\varphi(a)\theta(b\otimes c),
\end{equation*}
 and also
 \begin{equation*}
 \langle\theta(b\otimes c),\varphi\rangle=\langle\varphi(c)b,\varphi\rangle=\varphi(b)\varphi(c)=\varphi(bc)=\langle\pi(b\otimes c),\varphi\rangle.
 \end{equation*}
 Thus for each $a\in{\mathcal{A}}$ and $u\in{\mathcal{A}\hat{\otimes}\mathcal{A}}$
\begin{equation}\label{e2.1}
\theta(a\cdot u)=a\theta(u),\quad \theta(u\cdot a)=\varphi(a)\theta(u),
\end{equation}
and
\begin{equation}\label{e2.02}
\langle\theta(u),\varphi\rangle=\langle\pi(u),\varphi\rangle.
\end{equation}
By Goldstein's Theorem for every $F\in{(\mathcal{A}\hat{\otimes}\mathcal{A})^{**}}$  there is a bounded net $(u_\alpha)$ in $\mathcal{A}\hat{\otimes}\mathcal{A}$ such that $wk^*\hbox{-}\lim\limits_{\alpha}u_\alpha=F$. Since $\theta^{**}:(\mathcal{A}\hat{\otimes}\mathcal{A})^{**}\rightarrow\mathcal{A}^{**}$ is a $wk^*$-continuous map, by (\ref{e2.1}) for every $a\in{\mathcal{A}}$ we have
\begin{equation}\label{e2.2}
\theta^{**}(a\cdot F)=wk^*\hbox{-}\lim\limits_{\alpha}\theta^{**}(a\cdot u_\alpha)=wk^*\hbox{-}\lim\limits_{\alpha}a\cdot\theta^{**}(u_\alpha)=a\cdot\theta^{**}(F),
\end{equation}
and
\begin{equation}\label{e2.3}
\theta^{**}(F\cdot a)=wk^*\hbox{-}\lim\limits_{\alpha}\theta^{**}(u_\alpha\cdot a)=wk^*\hbox{-}\lim\limits_{\alpha}\varphi(a)\theta^{**}(u_\alpha)=\varphi(a)\theta^{**}(F),
\end{equation}
and (\ref{e2.02}) implies that
\begin{equation}\label{e2.4}
\langle\varphi,\theta^{**}(F)\rangle=\lim\limits_{\alpha}\langle\varphi,\theta^{**}(u_{\alpha})\rangle=\lim\limits_{\alpha}\langle\varphi,\pi_{\mathcal{A}}^{**}(u_{\alpha})\rangle=\langle\varphi,\pi_{\mathcal{A}}^{**}(F)\rangle.
\end{equation} 
For every $a\in{\mathcal{A}}$, $f\in{{\mathcal{A}}^{*}}$ and $u\in{\mathcal{A}\hat{\otimes}\mathcal{A}}$ by (\ref{e2.1}) we have
\begin{equation*}
\langle u,a\cdot\theta^*(f)\rangle=\langle u\cdot a,\theta^*(f)\rangle=\langle\theta(u\cdot a),f\rangle=\varphi(a)\langle\theta(u),f\rangle=\varphi(a)\langle u,\theta^*(f)\rangle,
\end{equation*}
and
\begin{equation*}
\langle u,\theta^*(f)\cdot a\rangle=\langle a\cdot u,\theta^*(f)\rangle=\langle\theta(a\cdot u),f\rangle=\langle a\theta(u),f\rangle=\langle\theta(u),f\cdot a \rangle=\langle u,\theta^*(f\cdot a)\rangle.
\end{equation*}
So  
\begin{equation}\label{e2.5}
a\cdot\theta^*(f)=\varphi(a)\theta^*(f),\quad \theta^*(f)\cdot a=\theta^*(f\cdot a)\quad(a\in{\mathcal{A}}, f\in{{\mathcal{A}}^{*}}).
\end{equation}
 Since $\varphi$ is $wk^*$-continuous, (\ref{e2.5}) implies that
  \begin{equation}\label{e2.6}
  \theta^*(\sigma wc(\mathcal{A}^*))\subseteq\sigma{wc}(\mathcal{A}\hat{\otimes}\mathcal{A})^{\ast}.
  \end{equation}
   Consider the quotient map $q:\mathcal{A}^{**}\rightarrow(\sigma wc(\mathcal{A}^*))^*$. By (\ref{e2.2}), (\ref{e2.3}) and (\ref{e2.6}) for every $a\in{\mathcal{A}}$ and $f\in{\sigma wc(\mathcal{A}^*)}$ we have
\begin{equation*}
\begin{split}
\langle f\cdot a,q\circ\theta^{**}(m_\alpha)\rangle&=\langle f\cdot a,\theta^{**}(m_\alpha)\vert_{\sigma{wc}({\mathcal{A}}^{\ast})}\rangle=\langle f\cdot a,\theta^{**}(m_\alpha)\rangle=\langle f, a\cdot\theta^{**}(m_\alpha)\rangle\\&=\langle f,\theta^{**}(a\cdot m_\alpha)\rangle=\langle \theta^*(f),a\cdot m_\alpha\rangle=\langle \theta^*(f),m_\alpha\cdot a\rangle=\langle f,\theta^{**}(m_\alpha\cdot a)\rangle\\&=\langle f,\varphi(a)\theta^{**}(m_\alpha)\rangle=\varphi(a)\langle f,\theta^{**}(m_\alpha)\vert_{\sigma{wc}({\mathcal{A}}^{\ast})}\rangle=\varphi(a)\langle f,q\circ\theta^{**}(m_\alpha)\rangle.
\end{split}
\end{equation*}
 Since $\varphi$ is $wk^*$-continuous, by \cite[Chapter V; Theorem 1.3]{Con:85}, $\varphi\in{\mathcal{A}_*}$. Note that $\mathcal{A}_*\subseteq\sigma wc(\mathcal{A}^*)$. Using (\ref{e2.4}) we have
\begin{equation}\label{e2.7}
\begin{split}
\langle \varphi,q\circ\theta^{**}(m_\alpha)\rangle&=\langle\varphi,\theta^{**}(m_\alpha)\vert_{\sigma{wc}({\mathcal{A}}^{\ast})}\rangle=\langle\varphi,\theta^{**}(m_\alpha)\rangle=\langle\varphi,\pi_{\mathcal{A}}^{**}(m_\alpha)\rangle\\&=\langle i_{\mathcal{A}_*}(\varphi),\pi_{\mathcal{A}}^{**}(m_\alpha)\rangle=\langle \varphi,i^*_{\mathcal{A}_*}\pi_{\mathcal{A}}^{**}(m_\alpha)\rangle.
\end{split}
\end{equation}
Since $i^{\ast}_{\mathcal{A}_{\ast}}\pi_{\mathcal{A}}^{**}(m_{\alpha}){a}\rightarrow{a}$ for every $a\in{\mathcal{A}}$ and $\varphi$ is continuous,
\begin{equation*}
\lim\limits_{\alpha}\langle\varphi,i^{\ast}_{\mathcal{A}_{\ast}}\pi_{\mathcal{A}}^{**}(m_{\alpha})\rangle\varphi(a)=\lim\limits_{\alpha}\langle\varphi,i^{\ast}_{\mathcal{A}_{\ast}}\pi_{\mathcal{A}}^{**}(m_{\alpha}){a}\rangle={\varphi(a)}.
\end{equation*}
(\ref{e2.7}) implies that $\lim\limits_{\alpha}\langle \varphi,q\circ\theta^{**}(m_\alpha)\rangle=1$ in $\mathbb{C}$. For sufficiently large $\alpha$, $\langle\varphi, q\circ\theta^{**}(m_\alpha)\rangle$ stays away from zero. Replacing $q\circ\theta^{**}(m_\alpha)$ by $\dfrac{q\circ\theta^{**}(m_\alpha)}{\langle\varphi,q\circ\theta^{**}(m_\alpha)\rangle}$, we may assume that
\begin{equation*}
\langle f\cdot a,q\circ\theta^{**}(m_\alpha)\rangle=\varphi(a)\langle f,q\circ\theta^{**}(m_\alpha)\rangle,\qquad\langle \varphi,q\circ\theta^{**}(m_\alpha)\rangle=1,
\end{equation*}
 for every $a\in{\mathcal{A}}$ and $f\in{\sigma wc(\mathcal{A}^*)}$. So $\mathcal{A}$ is $\varphi$-Connes amenable.
\end{proof}

\begin{Proposition}\label{p9}
	Let $\mathcal{A}$ and $\mathcal{B}$ be  dual Banach algebras. Suppose that $\theta:\mathcal{A}\longrightarrow\mathcal{B}$ is a continuous epimorphism which is also $wk^*$-continuous. If $\mathcal{A}$ is Johnson pseudo-Connes amenable, then $\mathcal{B}$ is Johnson pseudo-Connes amenable.
\end{Proposition}
\begin{proof}
Since $\mathcal{A}$ is Johnson pseudo-Connes amenable, there exists a net $(m_{\alpha})$ in $(\mathcal{A}\hat{\otimes}\mathcal{A})^{**}$ such that 
\begin{equation*}
\langle{T},a\cdot{m_{\alpha}}\rangle=\langle{T},{m_{\alpha}}\cdot{a}\rangle\quad\hbox{and} \quad i^{\ast}_{\mathcal{A}_{\ast}}\pi_{\mathcal{A}}^{**}(m_{\alpha}){a}\rightarrow{a},
\end{equation*}
 for every $a\in{\mathcal{A}}$ and $T\in\sigma{wc}(\mathcal{A}\hat{\otimes}\mathcal{A})^{\ast}$. Define $\theta\otimes\theta:\mathcal{A}\hat{\otimes}\mathcal{A}\rightarrow\mathcal{B}\hat{\otimes}\mathcal{B}$ by $\theta\otimes\theta(x\otimes y)=\theta(x)\otimes\theta(y)$, for every $x,y\in{\mathcal{A}}$. So $\theta\otimes\theta$ is a bounded linear map. For every $a,b,c\in{\mathcal{A}}$ we have
 \begin{equation*}
 \begin{split}
 (\theta\otimes\theta)(a\cdot (b\otimes c))&=\theta(ab)\otimes\theta(c)=\theta(a)\theta(b)\otimes\theta(c)\\&=\theta(a)\cdot(\theta(b)\otimes\theta(c))=\theta(a)\cdot(\theta\otimes\theta(b\otimes c)).
\end{split}
 \end{equation*}
 By similarity for the right action, for every $a\in{\mathcal{A}}$ and $u\in{\mathcal{A}\hat{\otimes}\mathcal{A}}$ we have
 \begin{equation}\label{e2.8}
 \theta(a)\cdot(\theta\otimes\theta)(u)=(\theta\otimes\theta)(a\cdot u),\qquad(\theta\otimes\theta)(u)\cdot\theta(a)=(\theta\otimes\theta)(u\cdot a).
 \end{equation}
  By Goldstein's Theorem for every $F\in{(\mathcal{A}\hat{\otimes}\mathcal{A})^{**}}$  there is a bounded net $(u_\alpha)$ in $\mathcal{A}\hat{\otimes}\mathcal{A}$ such that $wk^*\hbox{-}\lim\limits_{\alpha}u_\alpha=F$. Since $(\theta\otimes\theta)^{**}$ is a $wk^*$-continuous map, (\ref{e2.8}) implies that for every $a\in{\mathcal{A}}$
\begin{equation}\label{e2.9}
\theta(a)\cdot(\theta\otimes\theta)^{**}(F)=wk^*\hbox{-}\lim\limits_{\alpha}\theta(a)\cdot(\theta\otimes\theta)^{**}(u_{\alpha})=wk^*\hbox{-}\lim\limits_{\alpha}(\theta\otimes\theta)^{**}(a\cdot u_{\alpha})=(\theta\otimes\theta)^{**}(a\cdot F),
\end{equation}
and
\begin{equation}\label{e2.10}
(\theta\otimes\theta)^{**}(F)\cdot\theta(a)=wk^*\hbox{-}\lim\limits_{\alpha}(\theta\otimes\theta)^{**}(u_{\alpha})\cdot\theta(a)=wk^*\hbox{-}\lim\limits_{\alpha}(\theta\otimes\theta)^{**}(u_{\alpha}\cdot a)=(\theta\otimes\theta)^{**}(F\cdot a).
\end{equation}
Using (\ref{e2.8}) for every $a\in{\mathcal{A}}$, $u\in{\mathcal{A}\hat{\otimes}\mathcal{A}}$ and $f\in{(\mathcal{B}\hat{\otimes}\mathcal{B})^{*}}$ we have
\begin{equation*}
\begin{split}
\langle u,a\cdot(\theta\otimes\theta)^{*}(f)\rangle&=\langle u\cdot a,(\theta\otimes\theta)^{*}(f)\rangle=\langle\theta\otimes\theta(u\cdot a),f\rangle\\&=\langle(\theta\otimes\theta)(u),\theta(a)\cdot f\rangle=\langle u,(\theta\otimes\theta)^{*}(\theta(a)\cdot f)\rangle.
\end{split}
\end{equation*}
So by similarity for the right action, we have
\begin{equation}\label{e2.11}
a\cdot(\theta\otimes\theta)^{*}(f)=(\theta\otimes\theta)^{*}(\theta(a)\cdot f),\qquad (\theta\otimes\theta)^{*}(f)\cdot a=(\theta\otimes\theta)^{*}(f\cdot\theta(a)).
\end{equation}
Since $\theta$ is a $wk^*$-continuous map, (\ref{e2.11}) implies that
\begin{equation}\label{e2.12}
(\theta\otimes\theta)^{*}(\sigma wc(\mathcal{B}\hat{\otimes}\mathcal{B})^{*})\subseteq\sigma wc(\mathcal{A}\hat{\otimes}\mathcal{A})^{*}.
\end{equation}
 By (\ref{e2.9}), (\ref{e2.10}) and (\ref{e2.12}) for every $a\in{\mathcal{A}}$ and $U\in{\sigma wc(\mathcal{B}\hat{\otimes}\mathcal{B})^{*}}$ we have
\begin{equation*}
\begin{split}
\langle U,\theta(a)\cdot(\theta\otimes\theta)^{**}(m_\alpha)\rangle&=\langle U,(\theta\otimes\theta)^{**}(a\cdot m_\alpha)\rangle=\langle (\theta\otimes\theta)^{*}(U),a\cdot m_\alpha\rangle=\langle (\theta\otimes\theta)^{*}(U),m_\alpha\cdot a\rangle\\&=\langle U,(\theta\otimes\theta)^{**}( m_\alpha\cdot a)\rangle=\langle U,(\theta\otimes\theta)^{**}(m_\alpha)\cdot\theta(a)\rangle.
\end{split}
\end{equation*}
Using \cite[Chapter V; Theorem 1.3]{Con:85}, $\theta^{*}(\mathcal{B}_*)\subseteq\mathcal{A}_*$. For every $\psi\in{\mathcal{A}^{**}}$ and $h\in{\mathcal{B}_*}$ we have
\begin{equation*}
\langle i_{\mathcal{B}_{*}}^*\theta^{**}(\psi),h\rangle=\langle\theta^{**}(\psi),i_{\mathcal{B}_{*}}(h)\rangle=\langle\theta^{**}(\psi),h\rangle=\langle\psi,\theta^{*}(h)\rangle=\langle\psi,i_{\mathcal{A}_{*}}\theta^{*}(h)\rangle=\langle\theta i_{\mathcal{A}_{*}}^*(\psi),h\rangle,
\end{equation*}
 thus
 \begin{equation}\label{e2.13}
i_{\mathcal{B}_{*}}^*\theta^{**}=\theta i_{\mathcal{A}_{*}}^*.
 \end{equation}
 Since $\pi_{\mathcal{B}}\circ\theta\otimes\theta=\theta\circ\pi_{\mathcal{A}}$, (\ref{e2.13}) implies that for every $a\in{\mathcal{A}}$ 
\begin{equation*}
\lim\limits_{\alpha}(i_{\mathcal{B}_{*}}^*\pi_{\mathcal{B}}^{**}(\theta\otimes\theta)^{**}( m_\alpha))\theta(a)=\lim\limits_{\alpha}(\theta i_{\mathcal{A}_{*}}^*\pi_{\mathcal{A}}^{**}(m_\alpha))\theta(a)=\theta(\lim\limits_{\alpha}i_{\mathcal{A}_{*}}^*\pi_{\mathcal{A}}^{**}(m_\alpha)a)=\theta(a).
\end{equation*}
So $\mathcal{B}$ is Johnson pseudo-Connes amenable.
\end{proof}
\begin{cor}
		Let $\mathcal{A}$ be a dual Banach algebra and let $I$ be a $wk^*$-closed ideal of $\mathcal{A}$. If $\mathcal{A}$ is Johnson pseudo-Connes amenable, then ${\mathcal{A}}/{I}$ is Johnson pseudo-Connes amenable.
\end{cor}
\begin{proof}
Since the quotient map $Q:\mathcal{A}\rightarrow{\mathcal{A}}/{I}$ is a $wk^*$-continuous map, by Proposition \ref{p9} the dual Banach algebra ${\mathcal{A}}/{I}$ is Johnson pseudo-Connes amenable. 
	\end{proof}
%---------------------------------------------------------------------------------------------
%--------------------------------------------
\section{Some Applications}

\begin{Proposition}
	Let G be a locally compact group. The measure algebra $M(G)$ is Johnson pseudo-Connes amenable if and only if $G$ is amenable.
\end{Proposition}
\begin{proof}
	Let $M(G)$ be Johnson pseudo-Connes amenable. Since $M(G)$ is unital, by Lemma \ref{l5}, $M(G)$ is approximate Connes amenable. Hence $G$ is amenable \cite[Theorem 5.2]{Eslam:2012}.\\
	Conversely if $G$ is amenable, then $M(G)$ is Connes amenable \cite[Theorem 5.4]{Runde:2003}. Thus by Lemma \ref{l3}, $M(G)$ is Johnson pseudo-Connes amenable.
\end{proof}
Let $\mathcal{A}$ be a Banach algebra, $I$ and $J$ be arbitrary nonempty index sets
and let $P$ be a $J\times I$ matrix over $\mathcal{A}$ such that $\Vert P\Vert_{\infty}=\sup\{\Vert P_{j,i}\Vert:j\in{J},i\in{I}\}\leq1$. The set 
 of all $I\times J$ matrices over $\mathcal{A}$ with finite $\ell^1$-norm and product $XY=XPY$ is a Banach algebra, which is denoted by $LM(\mathcal{A},P)$ and it
is called the $\ell^1$-Munn $I\times J$ matrix algebra over $\mathcal{A}$ with sandwich matrix $P$ or briefly the $\ell^1$-Munn algebra \cite{eslam:04}.

 Suppose that $\mathcal{A}$ is a unital dual Banach algebra. Shojaee {\it et al.} showed that the $\ell^1$-Munn algebra $LM(\mathcal{A},P)=\ell^1(I\times J,\mathcal{A})$ is a dual Banach algebra with respect to the predual $c_0(I\times J,\mathcal{A}_*)$, where $I=J$ \cite{shoj:09}.
 
The Banach algebra of $I\times I$-matrices over $\mathbb{C}$, with finite $\ell^1$-norm and matrix multiplication is denoted by $\mathbb{M}_{I}(\mathbb{C})$, where $I$ is an arbitrary set. So $\mathbb{M}_{I}(\mathbb{C})$ is a dual $\ell^1$-Munn algebra over $\mathbb{C}$ with sandwich
matrix $P=id$.
\begin{Theorem}
	Let $I$ be a non-empty set. Then $\mathbb{M}_{I}(\mathbb{C})$ is Johnson pseudo-Connes amenable if and only if $I$ is finite.
\end{Theorem}	
\begin{proof}	
	Let $\mathcal{A}=\mathbb{M}_{I}(\mathbb{C})$ be Johnson pseudo-Connes amenable. Then there is a net $(m_{\alpha})$ in $(\mathcal{A}\hat{\otimes}\mathcal{A})^{**}$ such that $\langle{T},a\cdot{m_{\alpha}}\rangle=\langle{T},{m_{\alpha}}\cdot{a}\rangle$ and $i^{\ast}_{\mathcal{A}_{\ast}}\pi_{\mathcal{A}}^{**}(m_{\alpha})a\rightarrow{a}$ for every $a\in{\mathcal{A}}$ and $T\in\sigma\omega{c}(\mathcal{A}\hat{\otimes}\mathcal{A})^{\ast}$. Let $a$ be a non-zero element of $\mathcal{A}$. There exists a $\psi$ in $\mathcal{A}_{\ast}$ such that $a(\psi)\neq0$. Since $wk^*$-$\lim\limits_{\alpha}i^{\ast}_{\mathcal{A}_{\ast}}\pi_{\mathcal{A}}^{**}(m_{\alpha}){a}=a$, we have
	\begin{equation*}
	\lim\limits_{\alpha}\langle{a}\cdot\psi,i^{\ast}_{\mathcal{A}_{\ast}}\pi_{\mathcal{A}}^{**}(m_{\alpha})\rangle=\lim\limits_{\alpha}\langle\psi,i^{\ast}_{\mathcal{A}_{\ast}}\pi_{\mathcal{A}}^{**}(m_{\alpha})a\rangle=\langle\psi,a\rangle\neq0.
	\end{equation*}
	So we may assume that $\langle{a}\cdot\psi,i^{\ast}_{\mathcal{A}_{\ast}}\pi_{\mathcal{A}}^{**}(m_{\alpha})\rangle\neq0$. By Goldstein's theorem, there is a bounded net $(x^{\beta}_{\alpha})$ with bound $\Vert{m}_\alpha\Vert$ in $\mathcal{A}\hat{\otimes}\mathcal{A}$ such that $wk^{\ast}\hbox{-}\lim\limits_{\beta}{\hat{x}}^{\beta}_{\alpha}=m_{\alpha}$. It implies that for every $a\in{\mathcal{A}}$
	\begin{equation*}
	wk^{\ast}\hbox{-}\lim\limits_{\beta}a\cdot{\hat{x}}^{\beta}_{\alpha}=a\cdot{m_{\alpha}}\qquad\hbox{and}\qquad wk^{\ast}\hbox{-}\lim\limits_{\beta}{\hat{x}}^{\beta}_{\alpha}\cdot{a}={m_{\alpha}}\cdot{a}.
	\end{equation*}
	It follows that for every $T\in{\sigma\omega{c}(\mathcal{A}\hat{\otimes}\mathcal{A})^{\ast}}$,
	\begin{equation*}
	\begin{split}
	\lim\limits_{\beta}\langle{T},a\cdot q({\hat{x}}^{\beta}_{\alpha})-q({\hat{x}}^{\beta}_{\alpha})\cdot{a}\rangle=\lim\limits_{\beta}\langle{T},q(a\cdot{\hat{x}}^{\beta}_{\alpha}-{\hat{x}}^{\beta}_{\alpha}\cdot{a})\rangle&=\lim\limits_{\beta}\langle{T},a\cdot{\hat{x}}^{\beta}_{\alpha}-{\hat{x}}^{\beta}_{\alpha}\cdot{a}\rangle\\&=\langle{T},{a}\cdot{m_{\alpha}}-{m_{\alpha}}\cdot{a}\rangle=0,
	\end{split}
	\end{equation*}  
	where $q:(\mathcal{A}\hat{\otimes}\mathcal{A})^{\ast\ast}\rightarrow(\sigma{wc}({\mathcal{A}}\hat{\otimes}{\mathcal{A}})^{\ast})^{\ast}$ is the quotient map as in the Remark \ref{qou}. 
	So $wk^*$-$\lim\limits_{\beta}a\cdot q({\hat{x}}^{\beta}_{\alpha})-q({\hat{x}}^{\beta}_{\alpha})\cdot{a}=0$. Since $\pi_{\sigma{wc}}$ is a $wk^{\ast}$-continuous $\mathcal{A}$-bimodule morphism,
	\begin{equation}\label{e3.1}
	wk^*\hbox{-}\lim\limits_{\beta}a\pi_{\sigma{wc}}q({\hat{x}}^{\beta}_{\alpha})-\pi_{\sigma{wc}}q({\hat{x}}^{\beta}_{\alpha}){a}=0.
	\end{equation}
	For every $f\in{\mathcal{A}_{*}}$ we have
	\begin{equation*}
	\langle f,\pi_{\sigma{wc}}q({\hat{x}}^{\beta}_{\alpha})\rangle=\langle\pi^*\vert_{\mathcal{A}_{\ast}}(f),q({\hat{x}}^{\beta}_{\alpha})\rangle=\langle\pi^*(f),{\hat{x}}^{\beta}_{\alpha}\rangle=\langle f,\pi({x}^{\beta}_{\alpha})\rangle.
	\end{equation*}
	(\ref{e3.1}) implies that $	wk^*\hbox{-}\lim\limits_{\beta}a\pi({x}^{\beta}_{\alpha})-\pi({x}^{\beta}_{\alpha})a=0$. By Remark \ref{R1} and $wk^{\ast}$-continuity of $i^{\ast}_{\mathcal{A}_{\ast}}\pi_{\mathcal{A}}^{**}$ we have
	\begin{equation*}
	wk^*\hbox{-}\lim\limits_{\beta}\pi({x}^{\beta}_{\alpha})=wk^*\hbox{-}\lim\limits_{\beta}i^{\ast}_{\mathcal{A}_{\ast}}\pi_{\mathcal{A}}^{\ast\ast}({\hat{x}}^{\beta}_{\alpha})=i^{\ast}_{\mathcal{A}_{\ast}}\pi_{\mathcal{A}}^{\ast\ast}(m_{\alpha}).
	\end{equation*} 
	Let $y_{\beta}=\pi({x}^{\beta}_{\alpha})$. Then $(y_{\beta})$ is a bounded net in $\mathcal{A}$ which satisfies
	\begin{equation}\label{eq3.2}
	wk^*\hbox{-}\lim\limits_{\beta}ay_{\beta}-y_{\beta}a=0\quad\hbox{and}\quad	wk^*\hbox{-}\lim\limits_{\beta}y_{\beta}=i^{\ast}_{\mathcal{A}_{\ast}}\pi_{\mathcal{A}}^{\ast\ast}(m_{\alpha})\qquad(a\in{\mathcal{A}}).
	\end{equation}
	Suppose that $y_{\beta}=[y_{\beta}^{i,j}]$, where $y_{\beta}^{i,j}\in{\mathbb{C}}$ for every $i,j$. Fixed $i_{0}\in{I}$, for every $j\in{I}$ we have 
	\begin{equation*}
	\varepsilon_{i_0,j}y_{\beta}-y_{\beta}\varepsilon_{i_0,j}=\sum\limits_{\underset{i\neq{j}}{i\in{I}}}y_{\beta}^{j,i}\varepsilon_{i_0,i}+(y_{\beta}^{j,j}-y_{\beta}^{i_0,i_0})\varepsilon_{i_0,j}-\sum\limits_{\underset{i\neq{i_0}}{i\in{I}}}y_{\beta}^{i,i_0}\varepsilon_{i,j},
	\end{equation*}
	where $\varepsilon_{i,j}$ is a matrix belongs to $\mathbb{M}_{I}(\mathbb{C})$ which $(i,j)$-th entry is $1$ and others are zero. Let $$X_{\beta}=[X^{i,j}_\beta]=\varepsilon_{i_0,j}y_{\beta}-y_{\beta}\varepsilon_{i_0,j}.$$ 
	 (\ref{eq3.2}) implies that
	 $wk^*\hbox{-}\lim\limits_{\beta}X_{\beta}=0$. Consider $\varepsilon_{i_0,j}$, $\varepsilon_{i_0,i}$ as elements in $\mathcal{A}_*$, whenever $i\neq j$ in ${I}$. So
	\begin{equation}\label{eq3.3}
	\lim\limits_{\beta}y_{\beta}^{j,j}-y_{\beta}^{i_0,i_0}=\lim\limits_{\beta}X^{i_0,j}_\beta=\lim\limits_{\beta}\langle\varepsilon_{i_0,j},X_\beta\rangle=0,
	\end{equation}
	and
	\begin{equation}\label{eq3.4}
	 \lim\limits_{\beta}y_{\beta}^{j,i}=\lim\limits_{\beta}X^{i_0,i}_\beta=\lim\limits_{\beta}\langle\varepsilon_{i_0,i},X_\beta\rangle=0.
	\end{equation}
Since $\Vert{y}_{\beta}\Vert\leq\Vert{m_{\alpha}}\Vert$, $(y_{\beta}^{i_0,i_0})$ is a bounded net in $\mathbb{C}$. So it has a convergent subnet $(y_{\beta_{k}}^{i_0,i_0})$ in $\mathbb{C}$. We may assume  that $\lim\limits_{\beta_k}y_{\beta_{k}}^{i_0,i_0}=l$. By (\ref{eq3.3}) $\lim\limits_{\beta_k}y_{\beta_{k}}^{i_0,i_0}-y_{\beta_{k}}^{j,j}=0$. It follows that $\lim\limits_{\beta_{k}}y_{\beta_{k}}^{j,j}=l$ for every $j\in{I}$. If $l=0$, then by (\ref{eq3.4}) for every $i,j\in{I}$, $\lim\limits_{\beta_{k}}y_{\beta_{k}}^{i,j}=0$ in $\mathbb{C}$. So $wk$-$\lim\limits_{\beta_{k}}y_{\beta_{k}}^{i,j}=0$, where $i,j\in{I}$. Applying \cite[Theorem 4.3]{scha:71}, $wk$-$\lim\limits_{\beta_{k}}y_{\beta_{k}}=0$ in $\mathcal{A}$. It follows that $\lim\limits_{\beta_{k}}\langle{y}_{\beta_k},a\cdot{\psi}\rangle=0$. On the other hand
	\begin{equation*}
	\lim\limits_{\beta_{k}}\langle{y}_{\beta_k},a\cdot{\psi}\rangle=\lim\limits_{\beta_{k}}\langle{a}\cdot{\psi},{y}_{\beta_k}\rangle=\langle{a}\cdot{\psi},i^{\ast}_{\mathcal{A}_{\ast}}\pi_{\mathcal{A}}^{\ast\ast}(m_{\alpha})\rangle\neq0,
	\end{equation*}
	which is a contradiction. So $wk$-$\lim\limits_{\beta_{k}}y_{\beta_{k}}^{j,j}=l\neq0$ for every $j\in{I}$. Using (\ref{eq3.4}) we have $wk$-$\lim\limits_{\beta_k}y_{\beta_k}^{j,i}=0$ whenever $j\neq{i}$ in $I$. Applying \cite[Theorem 4.3]{scha:71} again, $wk$-$\lim\limits_{\beta_{k}}y_{\beta_{k}}=y_{0}$, where $y_0$ is a matrix with $l$ in the diagonal position and $0$ elsewhere. Thus $y_0\in{\overline{Conv(y_{\beta_k})}}^{wk}={\overline{Conv(y_{\beta_k})}}^{\Vert\cdot\Vert}$. So $y_0\in{\mathcal{A}}$. But 
	\begin{equation*}
	\infty=\sum\limits_{j\in{I}}\vert l\vert=\sum\limits_{j\in{I}}\vert y_0^{j,j}\vert=\Vert y_0\Vert<\infty,
	\end{equation*}
	which is a contradiction. So $I$ must be finite.\\
	Conversely, if $I$ is finite, then $\mathbb{M}_{I}(\mathbb{C})$ is Connes amenable \cite[Theorem 3.7]{Mah:2016}. So by Lemma $(\ref{l3})$, $\mathbb{M}_{I}(\mathbb{C})$ is 
	Johnson pseudo-Connes amenable. 
\end{proof}	
 Let $\mathcal{A}$ be a dual Banach algebra and let $I$ be a totally ordered set. Then the set of all $I\times{I}$-upper triangular matrices
with the usual matrix operations and the norm $\parallel[a_{i,j}]_{{i,j}\in{I}}\parallel=\sum\limits _{i,j\in{I}}\parallel{a}_{i,j}\parallel<\infty$, becomes a Banach algebra and it is denoted by
$$UP(I,\mathcal{A})=\set{\left[
	\begin{array}{rr} a_{i,j}  \end{array} \right]_{i,j\in I};  a_{i,j}\in A {\hbox{ and\, $a_{i,j}=0$ \,for every } }i>j }.$$
\begin{Theorem}\label{T7}
	Let $\mathcal{A}$ be a dual Banach algebra, $\varphi\in{\Delta_{wk^*}(\mathcal{A})}$ and let $I$ be a finite set. Then $\A$ is Johnson pseudo-Connes amenable if and only if $\mathcal{A}$ is Johnson pseudo-Connes amenable and $\vert I\vert=1$.
\end{Theorem}
\begin{proof}
	Let $\A$ be Johnson pseudo-Connes amenable. Assume that $I=\lbrace{i_{1},...,i_{n}}\rbrace$ and $\varphi$ in ${\Delta_{{wk}^{\ast}}{(\mathcal{A})}}$. We define a map $\psi:{UP(I,\mathcal{A})}\longrightarrow\mathbb{C}$ by          $\left[ a_{i,j}\right] _{i,j\in{I}}\longmapsto\varphi{(a_{i_{n},i_{n}})}$ for every $\left[ a_{i,j}\right] _{i,j\in{I}}\in{{UP(I,\mathcal{A})}}$.
	Since $\varphi$ is $wk^\ast$-continuous, $\psi\in{\Delta_{{wk}^{\ast}}{(UP(I,\mathcal{A}))}}$. By Proposition \ref{p6}, ${UP(I,\mathcal{A})}$ is $\psi$-Connes amenable. By similar argument as in \cite[Theorem 3.2]{Sha:2017}, we have $\lvert{I}\rvert=1$.
	
	Converse is clear.
\end{proof}
%%%%%%%%%%%%%%%%%%%%%%%%%%%%%%%%%%%%%%%%
%-------------------------------------------------------------------------------%%%%%%%%%%%%%%%%%%%%%%%%%%%%%%%%%%%%%%%%
\section{Examples}
\begin{Example}
		Consider the Banach algebra $\ell^1$ of all sequences $a=(a_n)$ of complex numbers with 
	\begin{equation*}
	\Vert{a}\Vert=\sum\limits_{n=1}^\infty\vert a_n\vert<\infty,
	\end{equation*}
	and the following product
	\begin{equation*}
	(a\ast b)(n)=\left\{
	\begin{array}{ll}
	a(1)b(1)& \hbox{if}\quad n=1\\
	a(1)b(n)+b(1)a(n)+a(n)b(n)&\hbox{if}\quad n>1
	\end{array}
	\right.
	\end{equation*}
	for every $a,b\in{\ell^1}$. It is easy to see that $\Delta(\ell^1)=\{\varphi_1\}\cup\{\varphi_1+\varphi_n:n\geq2\}$, where $\varphi_n(a)=a(n)$ for every $a\in{\ell^1}$. We claim that $(\ell^1,*)$ is a dual Banach algebra with respect to $c_0$. It is clear that $c_0$ is a closed subspace of $\ell^\infty$. We show that $c_0$ is an $\ell^1$-module with dual actions. For every $a\in{\ell^1}$ and $\lambda\in{c_0}$ we have
	\begin{equation*}
	a\cdot\lambda(n)=\left\{
	\begin{array}{ll}
	\sum\limits_{k=1}^\infty a(k)\lambda(k)& \hbox{if}\quad n=1\\
	(a(1)+a(n))\lambda(n)&\hbox{if}\quad n>1.
	\end{array}
	\right.
	\end{equation*}
	Since $\lambda$ vanishes at infinity and $\underset{n}{\sup}\vert a(n)\vert<\infty$, one can see that $a\cdot\lambda$ vanishes at infinity. So $a\cdot\lambda\in{c_0}$ and  similarity for the right action. We claim that $\ell^1$ is not Johnson pseudo-Connes amenable. Suppose conversely that $\ell^1$ is Johnson pseudo-Connes amenable. Since $\varphi_1$ is $wk^*$-continuous, by Proposition \ref{p6}, $\ell^1$ is $\varphi_1$-Connes amenable. Using \cite[Proposition 3.1]{Sha:2017} and by similar argument as in \cite[Theorem 3.2]{Sha:2017} there is a bounded net $(m_{\alpha})$ in $\ell^1$ that satisfies
\begin{equation}\label{e1}
a*{m_{\alpha}}-\varphi_1(a){m_{\alpha}}{\overset{wk^\ast}{\longrightarrow}}0\quad\hbox{and}\quad\varphi_1(m_{\alpha})\longrightarrow1\qquad({a}\in{\ell^1}).
\end{equation}
Choose $a=\delta_{n}$ in $\ell^1$, where $n\geq2$. So $\varphi_1(\delta_{n})=0$, using (\ref{e1}) we have  $\delta_{n}*{m_{\alpha}}{\overset{wk^\ast}{\longrightarrow}}0$ in $\ell^1$. One can see that  $\delta_{n}\ast m_\alpha=({m}_\alpha(1)+{m}_\alpha(n))\delta_{n}$. Consider $\delta_{n}$ as an element in $c_0$, where $n\geq2$. So
\begin{equation*}
\lim\limits_{\alpha}\langle\delta_{n},\delta_{n}\ast m_\alpha\rangle=\lim\limits_{\alpha}{m}_\alpha(1)+{m}_\alpha(n)=0.
\end{equation*}
Since $\lim\limits_{\alpha}m_{\alpha}(1)=1$ and $\lim\limits_{\alpha}m_{\alpha}(n)=-1$ for every $n\geq2$, we have $\underset{\alpha}{\sup}\Vert m_\alpha\Vert=\infty$, which  contradicts  the boundedness of the net $(m_{\alpha})$. 
\end{Example}
\begin{Example}
	Set $\mathcal{A}=\left(\begin{array}{cc} 0&\mathbb{C}\\
0&\mathbb{C}\\
\end{array}
\right)$. With the usual matrix multiplication and $\ell^{1}$-norm, $\mathcal{A}$ is a Banach algebra. Since $\mathbb{C}$ is a dual Banach algebra, $\mathcal{A}$ is a dual Banach algebra. Suppose that $\mathcal{A}$ is Johnson pseudo-Connes amenable. Then there is a net $(m_{\alpha})$ in $(\mathcal{A}\hat{\otimes}\mathcal{A})^{**}$ such that $\langle{T},a\cdot{m_{\alpha}}\rangle=\langle{T},{m_{\alpha}}\cdot{a}\rangle$ and $i^{\ast}_{\mathcal{A}_{\ast}}\pi_{\mathcal{A}}^{**}(m_{\alpha})a\rightarrow{a}$ for every $a\in{\mathcal{A}}$ and $T\in\sigma{wc}(\mathcal{A}\hat{\otimes}\mathcal{A})^{\ast}$. Let $M_{\alpha}=i^{\ast}_{\mathcal{A}_{\ast}}\pi_{\mathcal{A}}^{**}(m_{\alpha})$. By Remark \ref{R1} we have $i^{\ast}_{\mathcal{A}_{\ast}}\pi^{**}_{\mathcal{A}}=\pi_{\sigma{wc}}q$, where $q:(\mathcal{A}\hat{\otimes}\mathcal{A})^{\ast\ast}\longrightarrow(\sigma{wc}({\mathcal{A}}\hat{\otimes}{\mathcal{A}})^{\ast})^{\ast}$ is the quotient map as in the Remark \ref{qou}. Since $\pi_{\sigma{wc}}$ is an $\mathcal{A}$-bimodule morphism, for every $a\in{\mathcal{A}}$ we have
\begin{equation*}
\begin{split}
a M_{\alpha}&=a\pi_{\sigma{wc}}q(m_{\alpha})=\pi_{\sigma{wc}}q(a\cdot m_{\alpha})=\pi_{\sigma{wc}}(a\cdot m_{\alpha}\vert_{\sigma{wc}({\mathcal{A}}\hat{\otimes}{\mathcal{A}})^{\ast}})\\&=\pi_{\sigma{wc}}(m_{\alpha}\cdot a\vert_{\sigma{wc}({\mathcal{A}}\hat{\otimes}{\mathcal{A}})^{\ast}})=\pi_{\sigma{wc}}q(m_{\alpha}\cdot a)=\pi_{\sigma{wc}}q(m_{\alpha})a\\&=M_{\alpha}a.
\end{split}
\end{equation*}
So $(M_{\alpha})$ is a net in $\mathcal{A}$ satisfies $a M_{\alpha}=M_{\alpha}a$ and $M_{\alpha}a\rightarrow a$. Define the map $\phi:\mathcal{A}\longrightarrow\mathbb{C}$ by $$\phi\left(\begin{array}{cc} 0&a\\
0&b\\
\end{array}
\right)=b,\qquad(a,b\in{\mathbb{C}}).$$ It is clear that $\phi$ is linear and multiplicative and also for every $X,Y\in{\mathcal{A}}$ we have $XY=X\phi(Y)$. Choose $X\in{\mathcal{A}}$ such that $\phi(X)=1$. So $X=\lim\limits_{\alpha}M_{\alpha}X=\lim\limits_{\alpha}M_{\alpha}$. One can see that $X$ is a unit for $\mathcal{A}$. So for every $Y\in{\mathcal{A}}$ we have $Y=XY=X\phi(Y)$. So $\dim{\mathcal{A}}=1$, which is a contradiction.
\end{Example}

%%%%%%%%%%%%%%%%%%%%%%%%%%%%%%%%%%%%%%%%%%%%%
\begin{small}
	
\end{small}
\end{document}